\documentclass{article}
\usepackage{authblk}
\usepackage{xspace}
\usepackage[a4paper,outer=3cm]{geometry}

\usepackage{times}

\usepackage{xspace}

\usepackage{longtable}
\usepackage{amsmath,amsthm,amssymb,amsthm, stmaryrd}
\usepackage{enumerate}

\usepackage{algorithm}
\usepackage{algpseudocode}
\algnewcommand\algorithmicinput{\textbf{Input:}}
\algnewcommand\algorithmicoutput{\textbf{Output:}}
\algnewcommand\Input{\item[\algorithmicinput]}
\algnewcommand\Output{\item[\algorithmicoutput]}

\usepackage{hyperref}

\usepackage{subcaption}
\usepackage{pgf,tikz}
\usepackage{svg}
\usetikzlibrary{arrows}
\usetikzlibrary{decorations.markings}

\usepackage{pdftricks}
\begin{psinputs}
	\usepackage{pstricks,pst-plot,pst-node,pst-func}
\end{psinputs}
\usepackage{graphics,graphicx}

\tikzstyle{vertex}=[circle, draw, inner sep=0pt, minimum size=6pt]

\tikzset{->-/.style={decoration={
  markings,
  mark=at position .5 with {\arrow{>}}},postaction={decorate}}}

 \usepackage{marginnote}

\newcommand{\m}[1]{}

\usepackage[nothm]{thmbox}
\usepackage{amsthm}
\usepackage{thmtools}

\declaretheorem[parent=section,thmbox=M]{theorem}

\declaretheorem[numberlike=theorem,thmbox=M]{conjecture}

\newtheorem{claim}{Claim}[theorem]

\declaretheorem[numberlike=theorem]{lemma}
\declaretheorem[numberlike=theorem]{remark}
\declaretheorem[numberlike=theorem]{question}

\newenvironment{subproof}{\par\noindent {\it Proof}.\ }{\hfill$\blacklozenge$\par\vspace{11pt}}

\newcommand{\eN}{\mathbb{N}}

\newcommand{\sm}{\setminus}
\newcommand{\mc}{\mathcal}

\newcommand{\ra}{\rightarrow}
\newcommand{\Ra}{\Rightarrow}
\newcommand{\la}{\leftarrow}

\newcommand{\ora}[1]{\overrightarrow{#1}}

\DeclareMathOperator{\dic}{\ora \chi}

\newcommand{\F}{Forb_{ind}}

\newcommand{\overbar}[1]{\mkern 1.7mu\overline{\mkern-1.7mu#1\mkern-1.7mu}\mkern 1.7mu}

\tikzstyle{vertex}=[circle,draw, top color=gray!5, 
	    bottom color=gray!30, minimum size=12pt, scale=1, inner sep=0.5pt]
\tikzstyle{arc}=[->, > = latex',  thick]
\tikzstyle{edge}=[very thick, blue]

\newcommand{\omg}{oriented complete multipartite graph\xspace}

\newcommand{\omgs}{oriented complete multipartite graphs\xspace}

	{\noindent {}{#1}{}}{ This proves~(\arabic{claim}).\vspace{2ex}}

\title{Heroes in \omgs}

\author[1]{Pierre Aboulker}
\author[1,2]{Guillaume Aubian}
\author[2]{Pierre Charbit}

\affil[1]{DIENS, \'Ecole normale sup\'erieure, CNRS, PSL University, Paris, France.}
\affil[2]{Université de Paris, CNRS, IRIF, F-75006, Paris, France.}

\begin{document}

\maketitle

\begin{abstract}
The dichromatic number of a digraph is the minimum size of a partition of its vertices into acyclic induced subgraphs. Given a class of digraphs $\mathcal C$, a digraph $H$ is a hero in $\mc C$ if $H$-free digraphs of $\mathcal C$ have bounded dichromatic number. 
In a seminal paper, Berger at al. give a simple characterization of all heroes in tournaments. In this paper, we give a simple proof that heroes in quasi-transitive oriented graphs (that are digraphs with no induced directed path on three vertices) are the same as heroes in tournaments. We also prove that it is not the case in the class of oriented multipartite graphs, disproving a conjecture of Aboulker, Charbit and Naserasr, and give a characterisation of heroes in oriented complete multipartite graphs up to the status of a single tournament on $6$ vertices. 
\end{abstract}

\section{Introduction}
\subsection{Definitions and notations}
In this paper, we only consider \emph{directed graphs} (\emph{digraphs} in short) with no digons (a cycle on two vertices), loops nor multi-arcs. Let $G$ be a digraph. We denote by $V(G)$ its set of vertices and by $A(G)$ its set of arcs. 
For a vertex $x$ of $G$, we denote by $x^+$ (resp. $x^-$)  the set of its out-neighbours (resp. in-neighbours) and by $x^o$ the set of its non-neighbours with the convention that $x \notin x^o$). 
For a given set of vertices $X \subseteq V$, we denote by $G[X]$ the subgraph of $G$ induced by $X$.

Given two disjoint set of vertices $X, Y$ of a digraph $D$, we write $X \Rightarrow Y$ to say that for every $x \in X$ and for every $y \in Y$, $xy \in A(G)$, and we write $X \rightarrow Y$ to say that every arc with one end in $X$ and the other one in $Y$ is oriented from $X$ to $Y$ (but some vertices of $X$ might be non-adjacent to some vertices of $Y$). When $X=\{x\}$ we write $x \Rightarrow Y$ and $x  \rightarrow Y$. 

We also use the symbol $\Ra$ to denote a composition  operation on digraphs: for two digraphs $D_1$ and $D_2$, $D_1\Ra D_2$ is the digraph obtained from the disjoint union of $D_1$ and $D_2$ by adding all arcs from $V(D_1)$ to $V(D_2)$.

A  \emph{tournament} is an orientation of a complete graph. 
A {\em transitive tournament} is an acyclic tournament and we denote by $TT_n$ the unique acyclic tournament on $n$ vertices.  Given two tournaments $H_1$ and $H_2$, we denote by $\Delta(1,H_1,H_2)$ the tournament obtained from pairwise disjoint copies of $H_1$ and $H_2$ plus a vertex $x$, and all arcs from $x$ to the copy of $H_1$, all arcs from the copy of $H_1$ to the copy of $H_2$, and all arcs from the copy of $H_2$ to $x$.  When $\ell$ and $k$ are integers, we write $\Delta(1,k,H)$ for $\Delta(1, TT_k, H)$ and $\Delta(1,\ell,k)$ for $\Delta(1, TT_{\ell}, TT_k)$. The tournament $\Delta(1,1,1)$ is also denoted by $C_3$ and called a \emph{directed triangle}.

A \emph{$k$-dicolouring} of $G$ is a 
 partition of $V(G)$ into $k$ sets $V_{1}, \dots, V_{k}$ such that $G[V_{i}]$ is acyclic for $i = 1, \dots, k$. 
 The \emph{dichromatic number} of $G$, denoted by $\dic(G)$ and introduced by  Neuman-Lara~\cite{NL82}  is the minimum integer $k$ such that $G$ admits a $k$-dicolouring. 
We will sometimes extend $\dic$ to subsets of vertices, using $\dic(X)$ to mean $\dic(G[X])$ where $X \subseteq V$.

Given a set of digraphs $\mc H$, we say that a digraph $G$ is \emph{$\mc H$-free} if it contains no member of $\mc H$ as an induced subgraph. We denote by  $\F(\mc H)$ the class of $\mc H$-free digraphs. We write $\F(F_1, \dots, F_k)$ instead of $\F(\{F_1, \dots, F_k\})$ for simplicity. 
Given a class of digraphs $\mathcal C$, a digraph $H$ is a \emph{hero} in $\mathcal C$ if every $H$-free digraph in $\mathcal C$ has bounded dichromatic number. 

We denote by $\ora P_3$ the directed path on $3$ vertices.  
An \emph{\omg} is an orientation of a complete multipartite graph. Given two digraphs $G_1$ and $G_2$, $G_1 + G_2$ is the disjoint union of $G_1$ and $G_2$. We denote by $K_1$ the unique digraph on $1$ vertex. Observe that \omgs are precisely the digraphs in $\F(K_1 + TT_2)$.  

The main goal of this paper is to identify heroes in \omgs.  


\subsection{Context and results}

In a seminal paper, Berger et al.~\cite{hero} characterized heroes in tournaments:
\begin{theorem}[Berger et al.\ \cite{hero}]\label{thm:heroes}
A digraph $H$ is a hero in tournaments if and only if:
\begin{itemize}
\item $H=K_1$, or
\item $H=H_1 \Ra H_2$, where $H_1$ and $H_2$ are heroes in tournaments, or
\item $H=\Delta(1, k, H_1)$ or $H = \Delta(1, H_1, k)$, where $k\geq 1$ and $H_1$ is a hero in tournaments.
\end{itemize}
\end{theorem}

Observe that if a class of digraphs $\mathcal C$ contains all tournaments, then a hero in $\mc C$ must be a hero in tournaments. In \cite{ACN21}, it is conjectured that heroes in \omgs are the same as heroes in tournaments (actually a wider conjecture is proposed, see Section~\ref{sec:further_works}). We disprove this conjecture by showing  the following:  

\begin{theorem}\label{thm:ce} 
The digraphs $\Delta(1,2, C_3)$, $\Delta(1, C_3, 2)$, $\Delta(1,2, 3)$ and $\Delta(1,3, 2)$ are not heroes in \omgs.
\end{theorem}

On the positive side, we prove that:
\begin{theorem}\label{thm:main}
A digraph $H$ is a hero in \omgs if:
\begin{itemize}
    \item $H= K_1$, 
    \item $H=H_1 \Ra H_2$, where $H_1$ and $H_2$ are heroes in \omgs, or
    \item $H=\Delta(1, 1, H_1)$ where $H_1$ is a hero in \omgs.
\end{itemize}
\end{theorem}

Observe that the second bullet of the theorem above implies that a digraph is a hero in \omgs if \textit{and only if } each of its strong connected components are. Indeed, the \textit{only if} part of the assertion holds because an induced subgraph of a hero in any class is a hero in this class.  

Since a hero in \omgs must be a hero in tournaments,  Theorem~\ref{thm:heroes}, Theorem~\ref{thm:ce} and Theorem~\ref{thm:main} imply that, to get a full characterization of heroes in \omgs, it suffices to decide whether $\Delta(1,2,2)$ is a hero in \omgs or not. If it is not, then heroes in \omgs are precisely the ones described in Theorem~\ref{thm:main}. If it is, then a digraph $H$ is a hero in \omgs if and only if:
\begin{itemize}
    \item $H= K_1$ or $H=\Delta(1,2,2)$, 
    \item $H=H_1 \Ra H_2$, where $H_1$ and $H_2$ are heroes in \omgs, or
    \item $H=\Delta(1, 1, H_1)$ where $H_1$ is a hero in \omgs.
\end{itemize}

\begin{question}
Is $\Delta(1,2,2)$ a hero in \omgs?
\end{question}

\begin{remark}
    Between the submission of this paper and its acceptation,  Bartosz Walczak proved that non-interlaced ordered graphs (see Section~\ref{sec:122} for the definition) have unbounded chromatic number, which, together with  Theorem \ref{thm:interlaced_hero}, implies that $\Delta(1,2,2)$ is not a hero. According to the discussion above, this result  settles the question of characterizing the heroes in oriented complete multipartite graphs. We believe in the proof but since it is not yet officially reviewed and published, we preferred to not yet claim the complete Theorem.
\end{remark}

A digraph $G$ is  \emph{quasi-transitive} if for every triple of vertices $x,y,z$, if $xy,yz \in A(G)$, then $xz \in A(G)$ or $zx \in A(G)$. Observe that the class of quasi-transitive digraphs is precisely $\F(\ora P_3)$. Our last result is:
\begin{theorem}\label{thm:quasi_main} 
Heroes in quasi-transitive digraphs are the same as heroes in tournaments. 
\end{theorem}
\bigskip

\textbf{Organisation of the paper}:
We prove in Section~\ref{sec:counter_exemple} that $\Delta(1,2, C_3)$, $\Delta(1, C_3, 2)$, $\Delta(1,2, 3)$, $\Delta(1,3, 2)$ are not heroes in \omgs. We prove in Subsection~\ref{subsec:strong} that if $H_1$ and $H_2$ are heroes in \omgs, then so is $H_1 \Ra H_2$ and in subsection~\ref{subsec:growing_heroes} that if $H$ is a hero in \omgs, then so is $\Delta(1,1,H)$. We give some insight about whether $\Delta(1,2,2)$ should be a hero or not in \omgs in Section~\ref{sec:122} and finally, we prove Theorem~\ref{thm:quasi_main},  detail related results and propose some leads for further works in Section~\ref{sec:further_works}.


\section{Digraphs that are not heroes in \omgs}\label{sec:counter_exemple}

The goal of this section is to prove that $\Delta(1,2, C_3)$, $\Delta(1, C_3, 2)$, $\Delta(1,2, 3)$ and $\Delta(1,3, 2)$ are not heroes in \omgs. 
Since  reversing all arcs of a $\Delta(1,2, C_3)$-free \omg results in a $\Delta(1, C_3, 2)$-free \omg and does not change the dichromatic number, if $\Delta(1,2, C_3)$ is not a hero in \omgs then $\Delta(1,C_3, 2)$ is not either. 
Similarly, if $\Delta(1, 2, 3)$ is not a hero in \omgs then  $\Delta(1, 3, 2)$ is not either. 
Hence, it is enough to prove that  $\Delta(1,2, C_3)$ nor $\Delta(1,2, 3)$ are heroes in \omgs. This is implied by the existence of  $\{\Delta(1, 2, C_3), \Delta(1,2, 3)\}$-free \omgs with arbitrarily large dichromatic number. The rest of this section is dedicated to the description of such digraphs. \smallskip

A \emph{feedback arc set} of a given digraph $G$ is a set of arcs $F$ of $G$ such that their deletion from $G$ yields an acylic digraph. 
The idea of the construction comes from the fact that a feedback arc set of $\Delta(1,2, C_3)$ or of $\Delta(1,2,3)$ must induce a digraph with at least one vertex of in- or out-degree at least $2$. We then describe an \omg with large dichromatic number in which every subtournament has a feedback arc set inducing disjoint directed paths, implying that it does not contain $\Delta(1,2, C_3)$ nor $\Delta(1,2,3)$ by the fact above. 

\smallskip
Given an undirected graph $H$, a \emph{$k$-colouring} of $H$ is a partition of $V(G)$ into $k$ independent sets. The \emph{chromatic number} of $H$ is the minimum $k$ such that $H$ is $k$-colourable. 
Let $G$ be a digraph. We denote by $\chi(G)$ the chromatic number of the underlying graph of $G$. 
The (undirected) \emph{line graph} of $G$ is denoted by $L(G)$ and defined as follows: its vertex set is $A(G)$, and two of its vertices vertices $ab, cd \in A(G)$ are adjacent if and only if $b=c$. 

Be aware that the next lemma deals with chromatic number and not dichromatic number. We think it appears for the first time in~\cite{EH66}.

\begin{lemma}\cite{EH66}\label{lem:line_graph}
For every digraph $G$, we have $\chi(L(G)) \geq \log(\chi(G))$. 
\end{lemma}

\begin{proof}
Let $G$ be a digraph and assume $L(G)$ admits a $k$-colouring.
Observe that a colouring of $L(G)$ is the same as a colouring of the arcs of $G$ in such a way that no $\ora P_3$ is monochromatic. 
Consider the following colouring of $G$: for each $v \in V(G)$, colour $v$ with the set of colours received be the arcs entering in $v$. This is a $2^k$-colouring of $G$ because the colouring of $A(G)$ does not have monochromatic $\ora P_3$. 
\end{proof}

Let $s \geq 3$ be an integer and let us describe the graph $L(L(TT_s))$. Assuming the vertices of $TT_s$ are numbered $v_1, \dots, v_s$ in the topological ordering (that is, for all $1 \le i<j \le s$, we have $v_iv_j \in A(T)$), for any $i<j<k$, $\{v_i,v_j,v_k\}$ induces a $\ora P_3$ in $TT_s$. This way, we get a natural name for the vertices of $L(L(TT_s))$, namely $V(L(L(TT_s))) =  \{(v_i,v_j,v_k) \mid \text{ for every } i <j <k\}$. Moreover, edges of $L(L(TT_s))$ are of the form $(v_i,v_j,v_k)(v_j,v_k,v_{\ell})$ for  every $i<j<k<\ell$. For $2 \leq j \leq s-1$, set $V_j=\{(v_i,v_j,v_k)\}: i<j<k\}$. So $V_j$'s partition the vertices of  $L(L(TT_s))$ into stable sets.

We now define the digraph $D_s$ from $L(L(TT_s))$ as follows. 
The vertices of $D_s$ are the same as the vertices of  $L(L(TT_s))$ and  
$D_s$ is an \omg  with parts $(V_2, V_3, \dots, V_{s-1})$ and we orient the arcs as follow: given $j<k$, the edges of $L(L(TT_s))$ are oriented from $V_j$ to $V_k$ and all the other arcs are oriented from $V_k$ to $V_j$. This complete the description of $D_s$. 

The arcs $v_iv_j$ such that $i<j$ are called the \emph{forward arcs} of $D_s$, and the other arcs the \emph{backward arcs} of $D_s$. Observe that the underlying graph  induced by the forward arcs of $D_s$ is $L(L(TT_s))$. 

The following remark is the crucial feature of $D_s$. 
\begin{remark}\label{rmk:Ds}
Given a vertex $(v_i,v_j,v_k)$ of $D_s$, the out-neighbours of $(v_i,v_j,v_k)$ are all in $V_k$ and the in-neighbours of $(v_i,v_j,v_k)$ are all in $V_i$. 
\end{remark}


 Observe that a digraph that does not contain $\ora P_3$ as a subgraph is bipartite: all its vertices have in-degree $0$ or out-degree $0$, and the set of vertices with in-degree $0$ (resp. with out-degree $0$) form a stable set. 

\begin{lemma}\label{lem:Ds_dic}
For every integer $s$, $\dic(D_s) \geq \frac{1}{2}\log(\log(s))$. 
\end{lemma}

\begin{proof}
Let $V_2, \dots, V_{s-1}$ be the partition of $D_s$ as in the definition. Recall that $V(D_s) = \{(v_i,v_j,v_k): 1\leq i<j<k\leq s\}$. 
Denote by $F_s$ the digraph induced by the forward arcs of $D_s$. So the underlying graph of $F_s$ is $L(L(TT_s))$ and by Lemma~\ref{lem:line_graph}, $\chi(F_s) \geq log(log(s))$. 

Let $R$ be an acyclic induced subgraph of $D_s$. 
Observe that a directed path on $3$ vertices in $D_s$ using only arcs in $F_s$ must be of the form $(v_{i_1},v_{i_2},v_{i_3}) \rightarrow (v_{i_2},v_{i_3},v_{i_4}) \rightarrow (v_{i_3},v_{i_4},v_{i_5})$  
where $1\leq i_1<i_2<i_3<i_4<i_5 \leq s$ and is thus contained in a directed triangle of $D_s$ (because $(v_{i_1},v_{i_2},v_{i_3})(v_{i_3},v_{i_4},v_{i_5})$ is not an edge of $L(L(TT_s))$, and thus is not an arc of $F_s$, and thus $(v_{i_3},v_{i_4},v_{i_5})(v_{i_1},v_{i_2},v_{i_3})$ is an arc of $D_s$).  
Hence, the digraph with arcs $A(R) \cap A(F_s)$ does not contain $\ora P_3$ as a subgraph and is thus bipartite.    
Hence, a $t$-dicolouring of $D_s$ implies a $2t$-(undirected) colouring of $F_s$. As we have that $\chi(F_s) \geq log(log(s))$, the result follows. 
\end{proof}

\begin{lemma}\label{lem:feedback_edge_set}
If $T$ is a tournament contained in $D_s$, then $T$ has a feedback arc set formed by disjoint union of directed paths. 
\end{lemma}

\begin{proof}
Let $T$ be a subgraph of $D_s$ inducing a tournament. Then each vertex of $T$ belongs to a distinct $V_i$ and thus, by Remark~\ref{rmk:Ds}, the forward arcs of $D_s$ that are in $T$ induce a disjoint union of directed paths and clearly form a feedback arc set of $T$. 
\end{proof}

\begin{lemma}\label{lem:123}
For every $s \geq 1$, $D_s$ does not contain  $\Delta(1, 2, C_3)$ nor $\Delta(1, 2, 3)$. 
\end{lemma}

\begin{proof}
Observe that the two digraphs $\Delta(1, 2, C_3)$ and $\Delta(1, 2, 3)$ only differ on the orientation of one arc: reversing an arc of the copy of $C_3$ in $\Delta(1,2, C_3)$ leads to $\Delta(1,2, 3)$ and reversing an arc of the copy of $TT_3$ in $\Delta(1,2,3)$ leads to $\Delta(1,2, C_3)$. 
Our argument does not make any use of the orientations between the vertices inside this oriented $K_3$. Let $H$ be one of $\Delta(1, 2, C_3)$ or $\Delta(1, 2, 2)$, and let $x$ be the vertex in the copy of $K_1$, and  $y_1$ and $y_2$ the vertices in the copy of $TT_2$. See Figure~\ref{fig:123}.

Thanks to Lemma \ref{lem:feedback_edge_set}, it is enough to prove that in every feedback arc set of $H$, there exists a vertex with in- or out-degree at least $2$. Let $F$ be a feedback arc set of $H$ and assume for contradiction that it induces a disjoint union of directed paths. Then both $xy_1$ and $xy_2$ cannot belong to $F$. So we may assume without loss of generality that $xy_1 \notin F$. But then $F$ must intersect the three disjoint paths of length $2$ that go from $y_1$ to $x$, which necessarily implies that $F$ contains either two arcs coming out of $y_1$ or two arcs coming in $x$.
\end{proof}

\begin{figure}[ht]
\begin{center}
\begin{tikzpicture}[scale=1.5]

\node (x) at (2,2) [vertex] {$x$};
\node (y1) at (0.75,1) [vertex] {$y_1$};
\node (y2) at (1.5,0) [vertex] {$y_2$};
\node (z1) at (3,0) [vertex] {$z_1$};
\node (z2) at (4,0.25) [vertex] {$z_2$};
\node (z3) at (3.5,1) [vertex] {$z_3$};

\draw[arc] (x) to (y1);
\draw[arc] (x) to (y2);

\draw[arc] (y1) to (z1);
\draw[arc] (y1) to (z2);
\draw[arc] (y1) to (z3);
\draw[arc] (y2) to (z1);
\draw[arc, bend right] (y2) to (z2);
\draw[arc] (y2) to (z3);

\draw[arc] (z1) to (x);
\draw[arc, bend right=38] (z2) to (x);
\draw[arc] (z3) to (x);

\draw[edge] (y1) to (y2);

\draw[edge] (z1) to (z2);
\draw[edge] (z2) to (z3);
\draw[edge] (z3) to (z1);

\end{tikzpicture}
\end{center}
\caption{\label{fig:123} whatever the orientations of blue thick edges, $D_s$ does not contain this tournament and hence does not contain $\Delta(1,2,C_3)$ nor $\Delta(1,2,3)$.}

\end{figure}
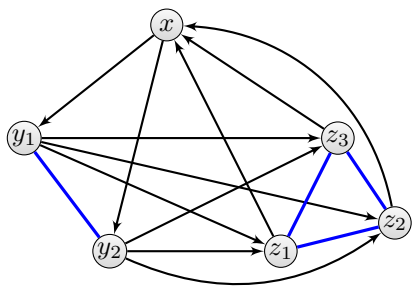

By Lemma~\ref{lem:Ds_dic} and Lemma~\ref{lem:123}, $\Delta(1,2,C_3)$ and $\Delta(1,2,3)$ are not heroes in \omgs.


\section{Heroes in \omgs}\label{sec:heroes}


\subsection{Strong components}\label{subsec:strong}

The goal of this subsection is to prove the following:

\begin{theorem}\label{thm:strong}
If $H_1$ and $H_2$ are heroes in $\F(K_1 + TT_2)$, then so is $H_1 \Ra H_2$. 
\end{theorem}

We actually prove the following stronger result:

\begin{theorem}\label{thm:strongg}
Let $H_1$, $H_2$ and $F$ be digraphs such that $H_1 \Ra H_2$ is a hero in $\F(F)$ and $H_1$ and $H_2$ are heroes in $\F(K_1 + F)$. Then $H_1 \Ra H_2$ is a hero in $\F(K_1 + F)$. 
\end{theorem}

To see that Theorem~\ref{thm:strongg} implies Theorem~\ref{thm:strong}, take $F=TT_2$ and observe that $\F(TT_2)$ is the class of digraphs with no arc and thus every digraph is a hero in $\F(TT_2)$.  We explain why such a stronger version can be of interest for future works in section~\ref{sec:further_works}

Note also that by taking $F=K_1$, we have that $\F(F)$ is empty and that $\F(K_1+F)$ is the class of tournaments, so Theorem~\ref{thm:strongg} yields the result of \cite{hero} (see \textbf{(3.1)}) stating that $H$ is a hero in tournaments if and only if all of its strong components are. Then, by induction, we get the same result for the class of digraphs with bounded independence number, reproving Theorem 1.4 of~\cite{HLNT19}. 
\medskip

The rest of this subsection is devoted to the proof of Theorem \ref{thm:strongg}, which is inspired but simpler (we got rid of the intricate notion of \emph{$r$-mountain}) than the analogous result for tournaments in \cite{hero}, even though our result is more general.

We start with a few definitions and notations. First, in order to simplify statements of the lemmas, we assume $H_1$, $H_2$ and $F$ are fixed all along the subsection and are as in the statement of Theorem~\ref{thm:strongg}. So there exists constants $c$ and $h$ such that:
\begin{itemize}
\item $H_1$ and $H_2$ have at most $h$ vertices,
\item digraphs in $\F(F, H_{1} \Rightarrow H_{2})$ have dichromatic number at most $c$, 
\item for $i=1,2$, digraphs in $\F(K_1+F, H_i)$ have dichromatic number at most $c$. 
\end{itemize}

If $G$ is a digraph and $uv \in A(G)$, we set $C_{uv} =v^+ \cap u^-$, that is the of vertices that form a directed triangle with $u$ and $v$. Finally, for $t\in \eN$, we say that a digraph $K$ is a \emph{$t$-cluster} if $\dic(K)\geq t$ and $|V(K)|\leq f(t)$, where $f(t)$ is the function defined recursively by $f(1)=1$ and $f(t)=1+f(t-1)(1+f(t-1))$.

The structure of the proof is very simple, we prove that digraphs in $\F(K_1+F, H_{1} \Rightarrow H_{2})$ that do not contain a $t$-cluster  have  dichromatic number bounded by a function of $t$ (Lemma~\ref{lemma:nocluster}), and then that the ones that contain a $t$-cluster  also have  dichromatic number bounded by a function of $t$ if $t$ is large enough (Lemma~\ref{lemma:cluster}). 

\begin{lemma}\label{lemma:nocluster}
There exists a function $\phi$ such that if $t$ is an integer and $G$ is a digraph in $\F(K_1+F, H_{1} \Rightarrow H_{2})$ which contains no $t$-cluster as a subgraph, then $\dic(G)\leq \phi(c,h,t)$
\end{lemma}

\begin{proof}
We prove this by induction on $t$. For $t=1$ the result is trivial as a $1$-cluster is simply a vertex. Assume the existence of $\phi(c,h,t-1)$, and assume $G$ is a digraph in $\F(K_1+F, H_{1} \Rightarrow H_{2})$ which contains no $t$-cluster. Say an arc $uv$ is \emph{heavy} if $C_{uv}$ contains a $(t-1)$-cluster, and \emph{light} otherwise. For a vertex $u$ we define $g(u)=\{v\in V(G) \mid  uv \text{ or } vu \text{ is a heavy arc}\}$. 

\begin{claim}
For any vertex $u$, $g(u)$ contains no $(t-1)$-cluster. 
\end{claim}

\begin{subproof}
Assume by contradiction that $K$ is a $(t-1)$-cluster in $g(u)$. By definition of $g(u)$, for every $v \in V(K)$, there exists a $(t-1)$-cluster $K_v$ in $C_{uv}$ or $C_{vu}$ (depending on which of $uv$ or $vu$ is an arc). Let $K'=\{u\}\cup V(K) \cup (\cup_{v\in K} V(K_v))$. We claim that $K'$ is a $t$-cluster. First note that the number of vertices of $K'$ is at most $1+f(t-1) +f(t-1) \cdot f(t-1)=f(t)$. We need to prove that $K'$ is not $(t-1)$-colourable, so let us consider for contradiction a $(t-1)$-colouring of its vertices, and without loss of generality assume $u$ gets colour $1$. Because $K$ is a $(t-1)$-cluster, some vertex $v$ in $K$ must also receive colour $1$, and since $K_v$ is also a $(t-1)$-cluster, some vertex $w$ in $K_v$ must also receive colour $1$, which produces a monochromatic directed triangle. So $K'$ is indeed a $t$-cluster, a contradiction. 
\end{subproof}

\begin{claim}
For any vertex $u$, $\min(\dic(u^{-}),\dic(u^{+})) \leq (h+1) \cdot (\phi(c,h,t-1)+c)$.
\end{claim}

\begin{subproof}
    Let $u \in V(G)$. By the previous claim and the induction hypothesis, $g(u)$ induces a digraph of dichromatic number at most $\phi(c,h,t-1)$, so it is enough to prove that one of the sets $u^-_{\ell}:=(u^{-}\setminus g(u))$ or $u^+_{\ell}:=(u^{+} \setminus g(u))$ induces a digraph with dichromatic number at most $h \cdot \phi(c,h,t-1)+c\cdot (h+1)$. 
    
    If $u^+_{\ell}$ induces an $H_2$-free digraph, then it has  dichromatic number at most $c< h\cdot\phi(c,h,t-1)+c\cdot (h+1)$, so we can assume that there exists $V_2 \subseteq u^{+}_{\ell}$ such that $G[V_2] = H_{2}$. We now cover $u^-_{\ell}$ with three sets $A, B, C$, each of which will have bounded dichromatic number.
    
    Let $A = u^{-}_{\ell} \cap (\cup_{v \in V_2} v^+) =u^{-}_{\ell} \cap (\cup_{v \in V_2} C_{uv})$.
    For every $v \in V_2$, $uv \in A(G)$ is light (because $V_2 \subseteq u_{\ell}^+$), so $G[C_{uv} \cap A]$ does not contain a $(t-1)$-cluster and is thus $\phi(c,h,t-1)$-colourable by induction. Now, since $H_2$ contains at most $h$ vertices, we get $\dic(A) \leq h \cdot  \phi(c,h,t-1)$.
        
    Let $B=u^{-}_{\ell} \cap (\cup_{v \in V_2} v^o)$. Since $G$ is $(K_1+F, H_1 \Ra H_2)$-free, for every $v \in V_2$, $v^o$ is $(F, H_1 \Ra H_2)$-free and thus $\dic(G[v^o]) \leq c$. Hence, $\dic(B) \leq c\cdot h$. 
    
    Finally, consider $C=u^{-}_{\ell}\setminus (A\cup B)$. By definition of $A$ and $B$, we get $C \Rightarrow V_2$. Since $G$ is $H_1 \Ra H_2$-free, $G[C]$ is $H_1$-free, and therefore $\dic(C) \leq c$.
    
    All together, we get $\dic(x^{-}_{\ell}) \leq h \cdot \phi(c,h,t-1)+c \cdot (h + 1)$ as desired.
\end{subproof}

By the previous claim, we can partition the set of vertices into the two sets $V^-$ and $V^+$ defined by: 
\begin{eqnarray*}
V^- = \{ u \in V \mid \dic(u^{-}) \leq (h+1) \cdot (c+\phi(c,h,t-1))\}\\
V^+ = \{ u \in V \mid \dic(u^{+}) \leq (h+1) \cdot (c+\phi(c,h,t-1))\}
\end{eqnarray*}

    If $G[V^-]$ is $H_1$-free and $G[V^+]$ is $H_2$-free, then $\dic(G) \leq 2c<\phi(c,h,t) $ and we are done. Assume  that there exists $V_1 \subseteq V^-$ such that $G[V_1] = H_{1}$ (the case where $V^+$ contains an induced copy of $H_2$ is symmetrical). 

    We now cover $V(G) \setminus V_1$ with three sets of vertices depending on their relation with $V_1$ and prove that each of these sets induces a digraph with bounded dichromatic number.
        
    Let $A=\bigcup_{v \in V_1}v^-$. By definition of $V^-$ and since $V_1 \subseteq V^-$, for every $v \in V_1$, $v^-$ has dichromatic number at most $(h+1)(c+\phi(c,h,t-1))$, and since $H_1$ has $h$ vertices we get that $\dic(A) \leq h \cdot (h+1) \cdot (c+\phi(c,h,t-1))$.
        
    Let $B=\bigcup_{v \in V_1}v^o$. Since $G$ is $(K_1+F, H_1 \Ra H_2)$-free, for every $v \in V_1$, $v^o$ is $(F, H_1 \Ra H_2)$-free and thus $\dic(G[v^o]) \leq c$. Hence, $\dic(B) \leq c \cdot h$.  
        
    Finally, let $C = V(G) \setminus (A\cup B \cup V_1)$. By definition of $A$ and $B$, we have $V_1 \Ra C$, hence $C$ is $H_2$-free and thus $\dic(C) \leq c$.  

    All together, we get that $\dic(G) \leq h + h \cdot (h+1) \cdot (c+\phi(c,h,t-1)) + ch +c =: \phi(c,h,t)$. 
\end{proof}

The proof of the theorem will follow from the second lemma below.

\begin{lemma}\label{lemma:cluster}
If $G\in \F(K_1+F, H_{1} \Rightarrow H_{2})$ contains a $(3c+1)$-cluster, then $\dic(G)\leq c \cdot 2^{f(3c+1)+1}$. 
\end{lemma}

\begin{proof}
Let $K$ be a $(3c+1)$-cluster in $G$. Assume there exists a vertex $u \in V(G)$ such that $u^-\cap V(K)$ is $H_1$-free and $u^+\cap V(K)$ is $H_2$-free. Since $u^o\cap V(K)$ is by assumption $(F , H_1 \Rightarrow H_2)$-free, we get a partition of $V(K)$ into three sets that induce digraphs with dichromatic number at most $c$, a contradiction (this still holds if $u\in K$ as we can add it to any of the sets without increasing the dichromatic number).

So, for every $u \in V(G)$, either $u^-\cap V(K)$ contains a copy of $H_1$, or $u^+\cap V(K)$ contains a copy of $H_2$. 
Now for every $V_1\subseteq V(K)$ such that $G[V_1]$ is isomorphic to $H_1$, the set of vertices $u$ such that $V_1 \subset u^-$ is $H_2$-free and therefore has dichromatic number at most $c$. 
Similarly, for every $V_2\subset V(K)$ such that $G[V_2]$ is isomorphic to $H_2$, the set of vertices $u$ such that $V_2 \subset u^+$ is $H_1$-free and therefore has dichromatic number at most $c$.
By doing this for every possible copy of  $H_1$ or $H_2$ inside $V(K)$ we can cover every vertex of $V(G)$. 
Moreover, the number of subsets of $V(K)$ that induces a copy of $H_1$ (resp. of $H_2$) is at most $2^{f(3c+1)}$. Hence, we get that $\dic(G)\leq c \cdot 2^{f(3c+1)+1}$. 
\end{proof}

\begin{proof}[of Theorem~\ref{thm:strongg}]
By Lemma~\ref{lemma:nocluster} and Lemma~\ref{lemma:cluster}, we get that every digraph in $\F(K_1+F, H_{1} \Rightarrow H_{2})$ has dichromatic number at most $\max(\phi(c,h,3c+1),2^{f(3c+1)+1} c)$, which proves Theorem~\ref{thm:strongg}.
\end{proof}

\begin{remark}
Let $K(c,h)$ an integer such that digraphs in $\F(K_1+F, H_{1} \Rightarrow H_{2})$ have dichromatic number at most $K(c,h)$. From the proof above we can deduce that taking 
$$K(c,h)=\max( (2h \cdot (h+1))^{5c+1} , 2 ^ {2 ^ {2 \cdot 3^{3c + 1}} + 1} \cdot c)$$ 
works (proving as intermediate steps that for every integer $t$, we can take $f(t) \leq 2 ^ {2 \cdot 3^t}$ and $\phi(c,h,t) \leq (2h \cdot (h+1))^{2c + t}$).
\end{remark}

\subsection{Growing a hero}\label{subsec:growing_heroes}

The goal of this subsection is to prove the following theorem:  
\begin{theorem}\label{thm:delta}
If $H$ is a hero in \omgs, then so is $\Delta(1, H, 1)$. 
\end{theorem}

The next lemma is proved in~\cite{hero} (see \textbf{(4.2)}) for tournaments but actually holds for every digraph. 
\begin{lemma}\label{lem:backedge}
Let $G$ a digraph and let $(X_1, \dots, X_n)$ a partition of $V(G)$. Suppose that $d$ is an integer such that:
\begin{itemize}
    \item $\forall\, 1 \leq i \leq n\ \dic(X_i) \leq d$ , %
    \item $\forall\, 1 \leq i < j  \leq n$ , if there is an arc $uv$ with $u \in X_j$ and $v \in X_i$, then $ \dic(X_{i+1} \cup X_{i+2} \cup \dots \cup X_j) \leq d $.
\end{itemize}
Then $\dic(G) \leq 2d$. 
\end{lemma}

\begin{proof}
Define a sequence $s_0<s_1<...<s_t=n$ defined recursively as follows: $s_0=0$ and 
$$s_{k}=\max\{ j>s_{k-1} \mid \dic(\bigcup_{s_{k-1} < i \leq j} X_i)\leq  d \}$$
for $k=1,\dots t$,  and let  $Y_k=\bigcup_{s_{k-1} < i \leq s_k} X_i$. 
By definition of the sequence $s_k$, $\dic(Y_k)\leq d$ for $k=1, \dots, t$ and $\dic(Y_k\cup X_{s_k+1})>d$ for $k=1, \dots, t-1$, so by the assumption of the lemma, there cannot be an arc from $Y_j$ to $Y_i$ whenever $i \leq j-2$.  Hence, $\bigcup_{i\ even} Y_i$ and $\bigcup_{i\ odd}Y_i$ both have dichromatic number at most $d$, and thus $\dic(G) \leq 2d$. 
\end{proof}

The following is an adaptation of \textbf{(4.4)} in~\cite{hero} with \omgs instead of tournaments  (note also that their proof is concerned with $\Delta(1, k, H)$ while ours is concerned with $\Delta(1, 1, H)$).

\begin{lemma}\label{lem:N+N-backedge}
Let $G$ be a $\Delta(1, 1, H)$-free  \omg given with a partition $(X_1, \dots, X_n)$ of its vertex set $V(G)$. Suppose that $r$ is an integer such that: 
\begin{itemize}
    
   \item $H$-free \omgs have dichromatic number at most $r$, 
   \item $\forall\, 1 \leq i \leq n\ \dic(X_i) \leq r$,
    \item $\forall\, 1 \leq i\leq n \ \forall v \in X_i \ \dic(v^+ \cap (X_1 \cup \dots \cup X_{i-1})) \leq r$,
    \item $\forall\, 1 \leq i\leq n \ \forall v \in X_i \ \dic(v^- \cap (X_{i+1} \cup \dots \cup X_{n})) \leq r$.
\end{itemize}
Then $\dic(G)\leq 8r+4$. 
\end{lemma}

\begin{proof} 
We are going to prove that $G$ satisfies the hypothesis of Lemma~\ref{lem:backedge} with $d = 4r+2$, which implies the result.   
Let $uv$ be an arc such that $u \in X_j$ and $v \in X_i$ where $1 \leq i < j \leq n$.
We want to prove that $ \dic(X_{i+1} \cup X_{i+2} \cup \dots \cup X_j) \leq 4r+2 $. 
Let $W = X_{i+1} \cup \dots \cup X_{j-1}$. Let $Q = v^+ \cap u^- \cap W$.  If $Q$ contains a copy of $H$, then together with $u$ and $v$ it forms a $\Delta(1, H, 1)$, a contradiction. So $Q$ is $H$-free and thus is $r$-colourable.
Now, each vertex in $W \setminus Q$ is in $u^+ \cup v^- \cup u^o \cup v^o$. 
By hypothesis, $u^+ \cap W$ and $v^- \cap W$ are both $r$-colourable, and since $G$ is an \omg, $u^o$ and $v^o$ are stable sets. Finally, by hypothesis, $\dic(X_j) \leq r$. All together, we get that $\dic( X_{i+1} \cup \dots \cup X_j) \leq 4r+2$ as announced. 
\end{proof}

\begin{proof}[of Theorem~\ref{thm:delta}]
Let $H$ be a hero in \omgs and let $h = |V(H)|$. By applying Theorem \ref{thm:strong} with $H_1 = H_2 = H$, we get that $H \Ra H$ is a hero in \omgs. Applying it again with $H_1 = H_2 = H \Ra H$, we get that $(H \Ra H) \Ra (H \Ra H)$ is a hero in \omgs. 
So there exists a constant $c$ such that every $\big((H \Ra H) \Ra (H \Ra H)\big)$-free \omg has dichromatic number at most $c$. Note that it also implies that every $H$-free  \omg has dichromatic number at most $c$. 

Let $G$ be a $\Delta(1, 1, H)$-free \omg. 
We are going to prove that $\dic(G) \leq 8r + 4$ for some $r$, using Lemma~\ref{lem:N+N-backedge}

We say that $J\subseteq V(G)$ is an \emph{$H$-jewel} if $G[J]$ is isomorphic to $H\Ra H$. 
The important feature about an $H$-jewel $J$ in an \omg is that, for any vertex $x$ not in $J$, either $x^+ \cap J$ or $x^- \cap J$ contains a copy of $H$, or $x$ has both an in- and an out-neighbour in $J$. 
An \emph{$H$-jewel-chain} of length $n$ is a sequence $(J_1, \dots, J_n)$ of pairwise disjoint $H$-jewels such that for  $i=1, \dots, n-1$, $J_i \Rightarrow J_{i+1}$, and for  every $1 \leq i < j \leq n$, $J_i \rightarrow J_j$. 
Both notions of $H$-jewel and $H$-jewel-chain exist in~\cite{hero}, the ones we give here are slightly different, but are morally similar.

Consider an $H$-jewel-chain $(J_1, \dots, J_n)$ of maximum length $n$. 
Set $J = J_1 \cup \dots \cup J_n$ and $W=V(G) - J$. To simplify statements, we also consider sets $J_i$ for $i\leq 0$ and $i \geq n+1$, that are assumed to be empty. 

The easy but key properties of an $H$-jewel-chain are stated in the following claim.

\begin{claim}\label{claim:jewel-chain}
For every $w\in W$ and $1\leq j \leq n-1$:
\begin{itemize}
\item $w^+\cap J_j \neq\emptyset \, \Ra \,  w^+\cap J_{j+1} \neq\emptyset$, 
\item $w^-\cap J_{j+1} \neq\emptyset\, \Ra \, w^-\cap J_{j} \neq\emptyset$. 
\end{itemize}

\end{claim}
\begin{subproof}
Assume $w^+\cap J_j \neq\emptyset$. Then since $J_j\Ra J_{j+1}$, it is not possible that $G[w^-\cap J_{j+1}]$ contains a copy of $H$ for it would create a $\Delta(1, H, 1)$. Since $G[J_{j+1}]$ is isomorphic to $H\Ra H$, and since $w$ cannot have a non neighbour in both copies of $H$ (because $G$ is an \omg), this implies that $w$ has at least one out-neighbour in $J_{j+1}$. The proof of the second item is identical up to reversal of the arcs.
\end{subproof}

For every $w \in W$, let $g(w)$ be the smallest integer $j$ such that $w^+\cap J_j \neq\emptyset$ if such an integer exists, and $g(w)=n+1$ if no such integer exists. 
For $j=1, \dots, n+1$, set $W_j = \{w:  g(w)=j\}$ and $X_j = J_j \cup W_j$. 
Note that, by definition of the $W_j$'s, if $w \in W_j$, then $J_i \ra w$ for every $i \leq j-1$. 
\vspace{0.3cm}

\begin{claim}\label{claim:X}
$\dic(X_j) \leq  4c \cdot h^2+ c + 6h $ for $j=1, \dots, n+1$. 
\end{claim}

\begin{subproof}
Let $1 \leq j \leq n+1$.
We have $\dic(J_j) \leq |J_j| \leq  2h$.  

 For each pair of vertices $a \in J_{j}$ and $b \in J_{j+1}$, 
set $A_{ab} = \{w \in W_j: bw, \, wa \in A(G)\}$. Since $ab \in A(G)$ (because $J_j \Rightarrow J_{j+1}$), and $G$ is $\Delta(1, H, 1)$-free, $A_{ab}$ must be $H$-free and thus is $c$-colourable for every choice of $a$ and $b$. 
Setting $A = \bigcup_{(a,b) \in J_j\times J_{j+1}} A_{ab}$, we get that $\dic(A) \leq 4 h^2 \cdot c$. Moreover, since every vertex in $W_j$ has an out-neighbour in $J_j$, we have $A = \{w \in W_j : w^- \cap J_{j+1} \neq \emptyset \}$.

Let $B = \{w \in W_j: w^o \cap J_{j-1} \neq \emptyset \text{ or }w^o \cap J_{j+1} \neq \emptyset \}$, in other words $B$ is the set of vertices in $W_j$ with at least one non-neighbour in $J_{j-1}$ or $J_{j+1}$. Since $G$ is an \omg, we have $\dic(B) \leq |J_{j-1}| + |J_{j+1}| \leq  4h$. 

Let $C = W_j \setminus (A \cup B)$. By definition of $W_j$, for every $i \leq j-1$, $J_i \rightarrow C$. 
Since $C$ is disjoint from $A$, we have $C \ra J_{j+1}$, and thus, by claim~\ref{claim:jewel-chain} (second bullet), we have $C \ra J_k$ for every $k \geq j+1$. 
Finally, since $C$ is disjoint from $B$, we have furthermore $J_{j-1} \Rightarrow C$ and $C \Rightarrow J_{j+1}$.  
Now, if $C$ contains an $H$-jewel-chain $(J'_1,J'_2)$ of length $2$, then $(J_1, \dots, J_{j-1}, J'_1, J'_2, J_{j+1}, \dots, J_n)$ is an $H$-jewel-chain of size $n+1$, contradicting the maximality of $n$. Hence, $C$ does not contain a jewel-chain of size $2$ and thus $\dic(C)\leq c$.

All together, we get that $\dic(X_j) \leq 4c \cdot h^2+ c + 6h$. 
\end{subproof}

\begin{claim}\label{clm:bounded_dichromatic_neighbours_J}
For $j=1, \dots, n$ and for every $u \in J_j$, 
\begin{itemize}
    \item $ \dic \big(u^+ \cap (X_1 \cup \dots \cup X_{j-1}) \big) \leq 4c \cdot h^2+  2c \cdot h + c + 6h $, and
    \item $ u^- \cap (X_{j+1} \cup \dots \cup X_{n+1}) = \emptyset $.
\end{itemize}
\end{claim}

\begin{subproof} 
Let $1 \leq j \leq n$ and let $u \in J_j$. 
We first prove the first  bullet.  
By definition of an $H$-jewel-chain, $u$ has no out-neighbor in any $J_i$ for $i \leq j-1$ and by Claim \ref{claim:X}, $\dic(X_{j-1}) \leq  4c \cdot  h^2+c+ 6h$. 
So it is enough to prove that $A =  u^+ \cap (W_1 \cup \dots \cup W_{j-2})$ has dichromatic number at most $2c \cdot h$. 
By Claim \ref{claim:jewel-chain}, every vertex of $W_1 \cup \dots \cup W_{j-2}$ has an out-neighbour in $J_{j-1}$.  
Moreover, for every $v \in J_{j-1}$, we have $vu \in A(G)$ (because $J_{j-1} \Ra J_j$)  and $v^- \cap A$ is $H$-free, for otherwise a copy of $H$ in $v^- \cap A$ would form, together with $v$ and $u$, a $\Delta(1, H, 1)$. So $\dic(A) \leq |J_{j-1}|\cdot c = 2c \cdot h$ as needed. 

To prove the second bullet, observe that for every $k \geq j+1$, since $J$ is a jewel-chain, $u$ has no in-neighbour in $J_k$ and by definition of $W_k$, $u$ has no in-neighbour in $W_k$. 
\end{subproof}

\begin{claim}\label{clm:bounded_dichromatic_neighbours_W}
For $j=1, \dots, n +1$ and for every $w \in W_j$, 
\begin{itemize}
    \item $ \dic\big(w^+ \cap (X_1 \cup \dots \cup 
X_{j-1}) \big) \leq 8c \cdot h^2+  2c \cdot h + 2c + 12h $, and
\item $ \dic\big(w^- \cap (X_{j+1} \cup \dots \cup X_{n+1}) \big) \leq 8c \cdot h^2 + 2c + 12h$.
\end{itemize}
\end{claim}

\begin{subproof}
Let $1 \leq j \leq n+1$ and let $w \in W_j$. 

We first prove the first  bullet. 
By definition of $W_j$, $w$ has no out-neighbor in any of the $J_i$ for $i\leq j-1$ and by Claim \ref{claim:X} $\dic(W_{j-2} \cup W_{j-1}) \leq  8c \cdot h^2 + 2c + 12h $. So it is enough to prove that $A= w^+ \cap \big(W_1 \cup \dots \cup W_{j-3}  \big)$ has  dichromatic number at most $2c \cdot h$. 
Again by definition of $W_j$ we have $J_{j-2}\rightarrow w$ and $J_{j-1} \rightarrow w$, and since $J_{j-2} \cup J_{j-1}$ induces a tournament and $G$ is $(K_1 + TT_2)$-free, $w$ has at most one non-neighbour in $J_{j-2} \cup J_{j-1}$. 
So there exists $s \in \{j-2, j-1\}$ such that $J_s \Rightarrow w$. 
For every $v \in J_s$, if $v^- \cap A$ contains a copy of $H$, then it would form, together with $v$ and $w$, a $\Delta(1, 1, H)$, a contradiction. So, for every $v \in J_s$, $v^- \cap A$ is $H$-free and is thus $c$-colourable. Finally, by claim \ref{claim:jewel-chain} every vertex in $A$ has an out-neighbour in $J_s$. So we get that $\dic(A) \leq 2c \cdot h$. 

We now prove the second bullet. 
If $j \geq n-1$, then by claim~\ref{claim:X} $\dic(X_n \cup X_{n+1}) \leq 8c \cdot h^2 + 2c + 12h$ and we are done. So we may assume that $j \leq n-2$. 
By claim~\ref{claim:X}, $\dic(X_{j+1}) \leq 4c \cdot h^2+ 
6h + c$.
Set $B = w^- \cap \big(X_{j+2} \cup \dots \cup X_{n+1} \big)$. 
By Claim \ref{claim:jewel-chain}, $w$ has an out-neighbour $v \in J_{j+1}$. For $i \geq j+2$, by definition of an $H$-jewel-chain, $v \rightarrow J_i$ and by definition of $W_i$, $v \rightarrow W_i$. So $v \rightarrow B$ and since $G$ is an \omg $B\setminus (v^+\cap B)$ is a stable set. Now, $v^+ \cap B$ is $H$-free, as otherwise $G$ would contain a $\Delta(1, H, 1)$. So $v^+ \cap B$ is $c$-colourable and thus $\dic(B) \leq c +1$ and thus $ \dic\big(w^- \cap (X_{j+1} \cup \dots \cup X_{n+1}) \big) \leq \dic(X_{j+1}) + c + 1 \leq 4c \cdot h^2 + 2c + 6h + 1$  by claim~\ref{claim:X}. 
\end{subproof}

By Claims~\ref{claim:X}, \ref{clm:bounded_dichromatic_neighbours_J} and ~\ref{clm:bounded_dichromatic_neighbours_W}, we can apply Lemma~\ref{lem:N+N-backedge} with  $r = 12c \cdot h^2 + 4c \cdot h + 3c + 18h$ to get $\dic(G) \leq 8r+4$. 
\end{proof}


\section{Some insights about $\Delta(1, 2, 2)$-free \omgs}\label{sec:122}

In~\cite{ARU18} Axenovich et al.\ tried to characterize patterns that must appear in every ordering of the vertices of graphs with large chromatic number. 
An (undirected) graph $G$ is (what we call) \emph{non-interlaced} if there exists an ordering $(x_1, \dots, x_n)$ on its vertices such that for every $i_1<i_2<i_3<i_4<i_5$, $\{x_{i_1}x_{i_3}, x_{i_3}x_{i_5}, x_{i_2}x_{i_4}\} \subsetneq E(G)$. See Figure~\ref{fig:non-interlaced}. They left as an open question whether non-interlaced graphs have bounded chromatic number or not. 
The goal of this section is to show that if $\Delta(1, 2, 2)$ is a hero in \omgs, then non-interlaced graphs have bounded chromatic number. See Theorem~\ref{thm:interlaced_hero}. 

 \begin{figure}[ht]
        \begin{center}
        \begin{tikzpicture}[scale=1.5]

        \node (v1) at (1,0) [vertex] {$x_{i_1}$};
        \node (v2) at (2,0) [vertex] {$x_{i_2}$};
        \node (v3) at (3,0) [vertex] {$x_{i_3}$};
        \node (v4) at (4,0) [vertex] {$x_{i_4}$};
        \node (v5) at (5,0) [vertex] {$x_{i_5}$};
        
        \draw[edge, bend right=38, black] (v3) to (v1);
        \draw[edge, bend right=38, black] (v5) to (v3);
        \draw[edge, bend right=38, black] (v4) to (v2);

        \end{tikzpicture}
        \end{center}
        \caption{\label{fig:non-interlaced} A graph is non-interlaced if there is an ordering of its vertices that avoids the above pattern as a subgraph.}
    \end{figure}
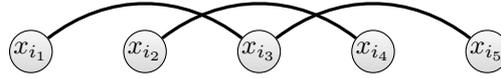

Given an \omg D together with an ordering $(V_1, \dots, V_n)$   on its parts, the arcs going from $V_i$ to $V_j$ are called  \emph{forward arcs} if $i<j$, and \emph{backward arcs} otherwise. Moreover, given $i<j$, we say that $u<v$ for every $u \in V_i$ and every $v \in V_j$. 
Finally, we say that an \omg $D$ is \emph{flat} if it admits an ordering $(V_1, \dots, V_n)$ on its parts such that for every vertex $v$ of $D$, the set of vertices $\{x\mid xv \text{ is a backward arc} \}$ is included in a single part of $D$, and the set of vertices $\{x\mid vx \text{ is a backward arc} \}$ is also included in a single part of $D$.  



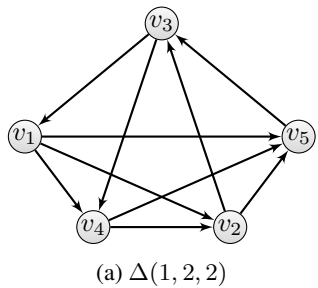
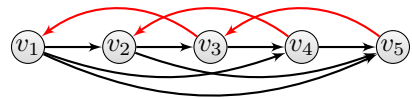
\begin{figure}
\centering
\begin{subfigure}{.5\textwidth}
  \centering
  \begin{tikzpicture}[scale=1.2] 
         \node (v3) at (2,2.25) [vertex] {$v_3$};
        \node (v1) at (0.5,1) [vertex] {$v_1$};
        \node (v4) at (1.25,0) [vertex] {$v_4$};
        \node (v2) at (2.75,0) [vertex] {$v_2$};
        \node (v5) at (3.5,1) [vertex] {$v_5$};
        
        \draw[arc] (v3) to (v1);
        \draw[arc] (v3) to (v4);
        
        \draw[arc] (v1) to (v5);
        \draw[arc] (v1) to (v2);
        \draw[arc] (v4) to (v2);
        \draw[arc] (v4) to (v5);
        
        \draw[arc] (v5) to (v3);
        \draw[arc] (v2) to (v3);
        
        \draw[arc] (v1) to (v4);
        \draw[arc] (v2) to (v5);
   \end{tikzpicture}
  \caption{$\Delta(1, 2, 2)$}
  \label{subfig:122}
\end{subfigure}%
\begin{subfigure}{.5\textwidth}
  \centering
  \begin{tikzpicture}[scale=1.2] 
          \node (v1) at (1,0) [vertex] {$v_1$};
        \node (v2) at (2,0) [vertex] {$v_2$};
        \node (v3) at (3,0) [vertex] {$v_3$};
        \node (v4) at (4,0) [vertex] {$v_4$};
        \node (v5) at (5,0) [vertex] {$v_5$};
        
        \draw[arc, bend right=38, red] (v3) to (v1);
        \draw[arc, bend right=38, red] (v5) to (v3);
        \draw[arc, bend right=38, red] (v4) to (v2);
        
        \draw[arc] (v3) to (v4);
        
        \draw[arc, bend left = -25] (v1) to (v5);
        \draw[arc] (v1) to (v2);
        \draw[arc] (v4) to (v5);
        
        \draw[arc] (v2) to (v3);
        
        \draw[arc, bend left = -20] (v1) to (v4);
        \draw[arc, bend left = -20] (v2) to (v5);
   \end{tikzpicture}
  \caption{A drawing of $\Delta(1, 2, 2)$ where the backward arcs (coloured in red) induce the forbidden pattern of non-interlaced graphs.}
  \label{subfig:122_line}
\end{subfigure}
\caption{Two drawings of $\Delta(1, 2, 2)$.}
\label{fig:122}
\end{figure}


\begin{lemma}\label{lem:flat_interlaced}
Let $D$ be an \omg with parts $V_1, \dots, V_n$ where $(V_1, \dots, V_n)$ is a flat ordering. If $D$ contains a copy of $\Delta(1, 2, 2)$, naming its vertices as in Figure~\ref{fig:122}, we must have $v_1 < v_2 < v_3 < v_4 < v_5$. 
\end{lemma}

\begin{proof}
    Suppose that $D$ contains a copy of $\Delta(1, 2, 2)$ and name its vertices as in Figure~\ref{fig:122}. 
    Since $\Delta(1, 2, 2)$ is a tournament, $v_i$'s are contained in pairwise distinct parts of $D$, and thus are totally ordered. 
    Since $(V_1, \dots, V_n)$ is a flat ordering, the smallest vertex among $\{v_1, v_2, v_3, v_4, v_5\}$ must have in-degree at most $1$ in $\Delta(1,2,2)$, and hence must be $v_1$. 
    Similarly, since $v_5$ is the only vertex with out-degree $1$ in $\Delta(1, 2, 2)$, $v_5$ must be the largest of the $v_i$.  
    If $v_3 < v_2$, then $v_3 < v_2 < v_5$ and the arcs $v_2v_3$ and $v_5v_3$ contradicts the fact that $(V_1, \dots, V_n)$ is a flat ordering, so $v_2 < v_3$. Similarly, if $v_4 < v_3$, then $v_4 < v_3 < v_5$ and the arcs $v_3v_4$ and $v_5v_3$ contradicts the fact that $(V_1, \dots, V_n)$ is a flat ordering, so  $v_3 < v_4$ and thus  $v_1 < v_2 < v_3 < v_4 < v_5$. 
\end{proof}

\begin{theorem}\label{thm:interlaced_hero}
If $\Delta(1, 2, 2)$ is a hero in \omgs, then every non-interlaced graph has bounded chromatic number.
\end{theorem}

\begin{proof}

Assume that $\Delta(1, 2, 2)$ is a hero in \omgs. 
Let $\mc F$ be the class of flat $\Delta(1, 2, 2)$-free \omgs. 
Since $\Delta(1, 2, 2)$ is a hero in \omgs, there exists a constant $r$ such that every digraph in $\mc F$ has dichromatic number at most $r$. Let $R \in \mc F$ such that $\dic(R) = r$ and recall that $R$ has a flat ordering. We are going to prove that every non-interlaced graph has chromatic number at most $2^{2^r}$.

Let $G$ be a non-interlaced (undirected) graph and $(x_1, \dots, x_n)$ the ordering on $V(G)$ given by the definition of non-interlaced graphs (that is an ordering that avoids the pattern in Figure~\ref{fig:non-interlaced}). We construct an \omg  $D'(G)$ as follow. 
 For each $x_i$, we create a stable set $V_i$ in $D'(G)$ of size $n^2$ and we assume the vertices of $V_i$ are organised as an $n \times n$ matrix. The parts of $D'(G)$ are $V_1, \dots, V_n$.  Let us now explain how we orient the arcs.  
Given $i<j$, if $x_ix_j \in E(G)$, we orient the arcs from each vertex of the $i^{th}$ line of $V_j$ to each vertex of the $j^{th}$ column of $V_i$.  Every other arc is oriented from $V_i$ to $V_j$. This completes the construction of $D'(G)$.

Let us now prove that the ordering $(V_1, \dots, V_n)$ of $D'(G)$ is flat. 
Let $v$ be any vertex of $D'(G)$ and assume $v \in V_j$ and is in the $i^{th}$ line and the $k^{th}$ column. By definition of $D'(G)$, if $v$ is the tail of some backward arcs $vw$, then $w$  belongs to the $j^{th}$ column of $V_i$ (in particular $i<j$). So all such $w$ belong to the same part. Similarly, if $uv$ is a backward arc, then $u$  belongs to the $j^{th}$-line of $V_k$ ($j<k$). This proves that $(V_1,  \dots, V_n)$ is a flat ordering of $D'(G)$.


We now construct another \omg $D(G)$ from $D'(G)$ by introducing, for $j=1, \dots, n-1$, a copy $R_j$ of $R$ between $V_j$ and $V_{j+1}$ such that $\cup_{i\leq j}V_i \Ra V(R_j)$, $V(R_j) \Ra \cup_{k \geq j+1}V_k$, and $V(R_j) \Ra \cup_{i\geq j+1}V(R_i)$ . This completes the construction of $D(G)$.

It is clear that $D(G)$ is an \omg and by inserting the flat ordering of each copy of $R$ between each consecutive $V_j$, we get a natural ordering of the parts of  $D(G)$. 
In the rest of the proof, we speak about backward and forward arcs of $D(G)$ with respect to this ordering. 

We are going to prove that $D(G) \in \mc F$ (so $\dic(D(G)) \leq r$) and that $\chi(G) \leq 2^{2^{\dic(D(G))}}$, which together imply the result. 

In order to help in our analysis, we will say that the vertices of $D(G)$ that comes from $D'(G)$ are green. 

The following claim is straightforward by construction. 
\begin{claim}\label{clm:backward_flat}
If $uv$ is a backward arc of $D(G)$, then either both $u$ and $v$ are green, or $u$ and $v$ are both contained in one of the copies of $R$.
\end{claim}

\begin{claim}\label{clm:122-interlaced}
If $v_1, v_2, v_3, v_4, v_5$ are vertices of $D'(G)$ such that $v_1 < v_2 < v_3 < v_4 < v_5$, then $\{v_3v_1, v_5v_3, v_4v_2\} \subsetneq A(D'(G))$. 
\end{claim}
\begin{subproof}
For otherwise $\{x_1x_3, x_3x_5,x_2x_4\} \subseteq E(G)$, a contradiction. 
\end{subproof}

Let us first prove that $D(G) \in \mc F$. 
By claim~\ref{clm:backward_flat}, $D(G)$ is flat and the ordering we consider is a flat ordering. Assume that $D(G)$ contains a copy of $\Delta(1, 2, 2)$ and name its vertices as in Figure~\ref{fig:122}. 
By Lemma~\ref{lem:flat_interlaced}, we have that the $v_i$ are in pairwise distinct parts of $D(G)$ and $v_1 < v_2 < v_3 < v_4 < v_5$.  
If $v_{3}$ is in a copy of $R$, since $v_3v_{1}$ and $v_{5}v_3$ are backward arcs of $D(G)$, we get by claim~\ref{clm:backward_flat} that $v_1$ and $v_5$ are in the same copy of $R$ as $v_3$. By construction, since $v_1<v_2<v_3<v_4<v_5$, we get that $v_2$ and $v_4$ are also in this same copy of $R$, a contradiction with the fact that $R$ is $\Delta(1, 2, 2)$-free. 
So we may assume that $v_{3}$ is green, and so are $v_{1}$ and $v_{5}$ by claim~\ref{clm:backward_flat}. Now, if $v_2$ is in a copy of $R$, then by claim~\ref{clm:backward_flat} $v_4$ is in the same copy of $R$, and since $v_2 < v_3 < v_4$,   $v_3$ must be in that same copy of $R$, a contradiction with the fact that $v_3$ is green.
Hence, $v_2$ is green and by claim \ref{clm:backward_flat} so is $v_4$. Thus, every $v_i$ is green, a contradiction to claim~\ref{clm:122-interlaced}.
This proves that $D(G) \in \mc F$. 
\smallskip

Since  $D(G)$ contains copies of $R$, it has dichromatic number at least $r$, and since $D(G) \in \mc F$, we get that $\dic(D(G)) = r$.  
 Consider a dicolouring $\ora \varphi$ of $D(G)$ with $r$ colours.  
We define a coloring $\varphi$ of $V(G)$ from $\ora \varphi$ as follows: for $i=1, \dots, n$, $\varphi(v_i)$ is the set of sets of colours used by each line of $V_i$. 
This gives us a colouring of $V(G)$ with at most $2^{2^r}$ colours. Let us prove that it is a proper colouring of $G$, that is, each colour class is an independent set. 

Assume for contradiction that there exists  $x_ix_j \in E(G)$ such that $\varphi(x_i) = \varphi(x_j)$ and assume without loss of generality that $i<j$. Let us first prove that $D(G)$ has a monochromatic backward arc. 
Consider the set of colours used in the $i^{th}$ line of $V_j$. The same set of colours is used by the vertices of some line of $V_i$, say the $k^{th}$. 
Now, there is an arc from each vertex of the $i^{th}$ line of $V_j$ to the  $j^{th}$ vertex of the $k^{th}$ line of $V_i$,
which implies the existence of a monochromatic backward arc as announced.   
Let $uv$ be this monochromatic backward arc, with $v \in V_i$ and $u \in V_j$. Since $i<j$, there is a copy of $R$ between $V_i$ and $V_j$. Since $\dic(R) = r$, one of the vertices $x$ of $R$ is coloured with $\ora \varphi(u)$. By construction of $D(G)$, $ux$ and $xv$ are arcs of $D(G)$ and thus $\{u,x, v\}$ induces a monochromatic directed triangle, a contradiction.   
\end{proof}

\section{Related and further works} \label{sec:further_works}

Heroes in orientations of chordal graphs was recently fully characterized in~\cite{AAS22}. \smallskip

A \emph{star} is an undirected tree with at most one non-leaf vertex. 
An \emph{oriented forest} (resp. \emph{oriented star}) is an orientation of a forest (resp. of a star). 
In~\cite{ACN21}, the authors initiated a systematic study of heroes in $\F(F)$ for a fixed digraph $F$. 
We now summarize the known results in this direction and explain how our results fit in the big picture. 

First observe that $K_1$ and $TT_2$ are heroes in every class of digraphs. 
A result in~\cite{HM12} implies that no digraph except for $K_1$ and $TT_2$ is a hero in $\F(F)$ whenever the underlying graph of $F$ contains a cycle. 
We now distinguish cases depending on whether $F$ is an oriented forest, an oriented star or a disjoint union of at least two oriented stars. 

\subsection{Heroes in $\F(F)$ when $F$ is an oriented forest}

It is  proved in~\cite{ACN21} that if $F$ is not a disjoint union of oriented stars, then the only possible heroes in $\F(F)$ are transitive tournaments. 
In the same paper the authors venture to conjecture the following (which can be seen as an oriented analogue of the well-known Gy\'arf\'as-Sumner  conjecture~\cite{G75, S81}): 

\begin{conjecture}[\cite{ACN21}]
For every oriented forest $F$, every transitive tournament is a hero in $\F(F)$.
\end{conjecture}

In~\cite{S21} it is proved that it is enough to prove the conjecture for trees, the conjecture have been proved to be true for oriented stars~\cite{CS19}. 

\subsection{Heroes in $\F(F)$ when $F$ is an oriented star}

When $F$ is an oriented star, it is still possible that heroes in $\F(F)$ are the same as heroes in tournaments.  As said in the previous subsection, it is proved in~\cite{CS19} that for every oriented star $F$, all transitive tournaments are heroes in $\F(F)$.  The only other known result so far is  concerned with $\ora K_{1,2}$ (the oriented star on $3$ vertices, with one vertex of out-degree $2$ and two vertices of in-degree $1$): it is proved in~\cite{AAC21, S21} that $K_1 \Ra C_3$ (and thus $C_3$ too) is a hero in $\F(\ora K_{1,2})$. Note that $\ora P_3$ is an oriented star. 
We now give an easy proof that all heroes in tournaments are heroes in $\F(\ora P_3)$. 
\smallskip

Recall that a digraph $G$ is  \emph{quasi-transitive} if for every triple of vertices $x,y,z$, if $xy,yz \in A(G)$, then $xz \in A(G)$ or $zx \in A(G)$ and observe that the class of quasi-transitive digraphs is precisely $\F(\ora P_3)$. 

Given two digraphs $G_1$ and $H_1$ with disjoint vertex sets, a vertex $u \in G_1$, and a digraph $G$, we say that $G$ is obtained by substituting $H_1$ for $u$ in $G_1$, and write $G_1(u \la H_1)$ to denote $G$, 
provided that the following hold: 
\begin{itemize}
    \item $V(G) = (V(G_1) \setminus u) \cup V(H_1)$,
    \item $G[V(G_1) \setminus u] = G_1 \setminus u$,
    \item $G[V(H_1)] = H_1$
    \item for all $v \in  V(G_1) \setminus u$ if $vu \in A(G_1)$ (resp. $uv \in A(G_1)$, resp. $u$ and $v$ are non-adjacent in $G_1$), then $V(H_1) \Ra v$  (resp. $v \Ra V(H_1)$, resp. $V(H_1) \subseteq v^o$) in $G$. 
\end{itemize}
Let $\mc T$ be the class of tournaments and $\mc A$ the class of acyclic  digraphs. Let $(\mc A \cup \mc T)^*$ be the closure of $\mc A \cup \mc T$ under taking substitution, that is to say digraphs in $(\mc A \cup \mc T)^*$ are the digraphs obtained from a vertex by repeatedly substituting vertices by digraphs in $\mc A \cup \mc T$.
A classic result of Bang-Jensen and Huang~\cite{BH95} (see also Proposition 8.3.5 in~\cite{BG18}), implies that quasi-transitive digraphs are all in $(\mc A \cup \mc T)^*$.

\begin{theorem}
Heroes in $(\mc A \cup \mc T)^*$ are the same as heroes in tournaments. In particular, heroes in $\F(\ora P_3)$ are the same as heroes in tournaments. 
\end{theorem}

\begin{proof}
Let $H$ be a hero in tournaments and $c$ be the maximum dichromatic number of an $H$-free tournament. 
We prove by induction on the number of vertices that $H$-free digraphs in $(\mc A \cup \mc T)^*$ are also $c$-dicolourable. 
Let $G \in \mc (\mc A \cup \mc T)^*$ on $n \geq 2$ vertices and assume that all digraphs in $\mathcal (\mc A \cup \mc T)^*$ on at most $n-1$ vertices are $c$-dicolourable. 

There exist $G_1, \dots, G_s$, $H_1, \dots, H_{s-1}$ and vertices $v_1 \dots, v_{s-1}$ such that  the $G_i$'s and the $H_i$'s are  digraphs of $\mc A \cup \mc T$ with at least two vertices, $G_1 = K_1$, $G_s = G$, $v_i \in V(G_i)$ and for $i=1, \dots, s-1$, $G_{i+1} = G_{i}(v_{i} \la H_{i})$. 

If all $H_i$ are tournaments, then $G$ is a tournament and is thus $c$-dicolourable. 
So we may assume that there exists $1 \leq i \leq s-1$ such that $H_i$ is an acyclic digraph. 
Let $x_1, \dots, x_t$ be the vertices of $H_i$. 
There exist $t$ digraphs $X_1, \dots, X_t$ in $\mathcal (\mc A \cup \mc T)^*$  such that $G$ is obtained from $G_{i+1}$ by substituting $x_1$ by $X_1$, $x_2$ by $X_2$, $\dots,$ $x_t$ by $X_t$  and some vertices of $V(G_{i+1}) \setminus \{x_1, \dots, x_t\}$ by digraphs in $(\mc A \cup \mc T)^*$.  Note that the order in which  these substitutions are performed does not matter. 

Let $X = \cup_{1\leq i \leq  t} V(X_i)$. 
So $V(G) \sm X$ can be partitioned into $3$ sets $S^+$, $S^-$, $S^o$ such that for every $v \in X$, $S^+ \subseteq v^+$, $S^- \subseteq v^-$ and $S^o \subseteq v^o$. 

For $i=1, \dots, t$, let $D_i = G[G_i \setminus (X \setminus X_i)]$. 
By induction, the $D_i$'s are $c$-dicolourable. For $i=1, \dots, t$, let $\phi_i$ be a $c$-dicolouring of $D_i$. 
Assume without loss of generality that $|\phi_1(X_1)| \geq |\phi_i(X_i)|$ for $1 \leq i \leq t$. 
In particular $\dic(X_i) \leq |\phi_1(X_1)|$ for $i=1, \dots, t$. 
Extend $\phi_1$ to a $c$-dicolouring of $D$ by dicolouring each $X_i$ (independently) with colours from $\phi_1(X_1)$. We claim that this gives a $c$-dicolouring of $G$. 

Let $C$ be an induced directed cycle of $G$. If $C$ is included in $X$ or $V(G) \sm X$, then $C$ is not monochromatic. 
So we may assume that $C$ intersects both $V(G) \sm X$ and $X$. 
Since vertices in $X$ share the same neighborhood outside $X$ and $C$ is induced, $C$ must intersect $X$ on exactly one vertex, and this vertex can be chosen to be any vertex of $X$. In particular we may assume that it is in $X_1$. Hence $C$ is not monochromatic. 
\end{proof}

Note that the proof of the previous theorem actually works for the following stronger statement:
\begin{theorem}
Let $\mc C$ be a class of digraphs closed under taking substitution and let $(\mc A \cup \mc C)^*$ be the closure of $\mc A \cup \mc C$ under taking substitution. Then heroes in $(\mc A \cup \mc C)^*$ are the same as heroes in $\mc C$. 
\end{theorem}

\subsection{Heroes in $\F(F)$ when $F$ is a disjoint union of at least two oriented stars}

When $F$ is a disjoint union of stars, the authors of~\cite{ACN21} conjectured that heroes in $\F(F)$ were the same as heroes in tournaments, and Theorem~\ref{thm:ce} disproves this conjecture (recall that $\F(K_1 + TT_2)$ is the class of \omgs). 

Since $\F(F_1) \subseteq \F(F_2)$ whenever $F_1$ is an induced subgraph of $F_2$, and given our knowledge on heroes in $\F(F)$ when $F$ is an oriented star, let us focus on disjoint union of stars where each connected component is $K_1$, $TT_2$ or $\ora{P}_3$. 

We denote by $\overbar{K}_t$ the digraph on $t$ vertices with no arc (this is a disjoint union of stars, where each connected component is $K_1$). Observe that $\F(\overbar{K}_2)$ is the class of tournaments. 
In~\cite{HLNT19}, it is proved that heroes in $\F(\overbar{K}_t)$ are the same as heroes in tournaments. 
The proof of this result is quite hard, and shows that knowing heroes in $\F(F)$  does not necessarily help in understanding heroes in $\F(K_1 + F)$. Even worse, it is clear that every digraph is a hero in $\F(K_1)$ and in $\F(TT_2)$, while our result shows that only very few digraphs are heroes in $\F(K_1 + TT_2)$. 

Theorem~\ref{thm:strongg} suggests that the heroes in $\F(K_1 + TT_2)$ could be the same as heroes in $\F(F)$ where $F = \overbar{K}_t + TT_2$ or $F = \overbar{K}_t + \ora{P}_3$.  
In order to prove it (up to the status of $\Delta(1,2,2)$), it would be enough to answer by the affirmative to the following question:

\begin{question}
Let $H$ and $F$ be digraphs such that $\Delta(1,1,H)$ is a hero in $\F(F)$ and $H$ is a hero in $\F(K_1 + F)$. Then $\Delta(1,1,H)$ is a hero in $\F(K_1 + F)$.
\end{question}
\bigskip

{\bf Acknowledgments:} 
This research was partially supported by ANR project DAGDigDec (JCJC)   ANR-21-CE48-0012 and by the group Casino/ENS Chair on Algorithmics and Machine Learning.

\end{document}